\documentclass[a4paper, 12pt]{article}
\usepackage{hyperref}
\usepackage[utf8]{inputenc}
\usepackage[T2A]{fontenc}
\usepackage{amsmath,amssymb,amsthm}
\usepackage[a4paper,hmargin=2.5cm,vmargin=2.5cm]{geometry}
\usepackage[english]{babel}
\usepackage{titlefoot}

\newcommand{\RR}{\mathbb{R}}
\newcommand{\ZZ}{\mathbb{Z}}
\newcommand{\QQ}{\mathbb{Q}}

\newcommand{\PP}{\mathbb{P}}
\newcommand{\OO}{\mathcal{O}}
\newcommand{\restr}[1]{|_{#1}}
\newcommand{\Isom}{\textnormal{Isom}}
\newcommand{\Ext}{\textnormal{Ext}}

\newtheorem{proposition}{Proposition}
\newtheorem{definition}{Definition}
\newtheorem{remark}{Remark}
\newtheorem{lemma}{Lemma}
\newtheorem{conjecture}{Conjecture}

\title{The group of isometries of $K_0(\PP_n)$}

\author{Ivan Beldiev}

\newcommand{\Addresses}{{
  \bigskip
  \footnotesize

  \textsc{HSE University, Faculty of Computer Science, Pokrovsky Boulevard 11, Moscow,
109028 Russia}\par\nopagebreak
  \textit{E-mail address}: \texttt{isbeldiev@hse.ru}

}}

\begin{document}

\date{}

\maketitle

\unmarkedfntext{2010 Mathematics Subject Classification. Primary 14F05, 15A63, 19E08; Secondary 20G07. \newline\hspace*{2em}Key words and phrases. Projective space, coherent sheaves, Grothendieck group, Euler form, exceptional basis, group of isometries.}
\begin{abstract} We study the group of isometries of the Grothendieck group $K_0(\PP_n)$ equipped with the standard Euler form $\chi$ defined by $\chi(E, F) = \sum_{\nu}(-1)^\nu\dim \Ext^\nu(E, F)$. We prove several properties of this group, in particular, we show that it is essentially a free abelian group of rank $[\frac{n+1}{2}]$. Also, we compute explicitly its generators for~$n\leqslant 6$.\end{abstract}

\section*{Introduction}\

The Grothendieck group $K_0(\PP_n)$ of coherent sheaves on $\PP_n$ is a free $\ZZ$-module of rank $n+1$ equipped with the Euler form

$$\chi(E, F) = \sum_{\nu}(-1)^\nu\dim \Ext^\nu(E, F).$$

From now on, for a coherent sheaf $E$ we denote by the same letter $E$ its class in~$K_0(\PP_n)$.

\begin{definition}
    A basis $E_0, E_1, \ldots, E_n$ of $K_0(\PP_n)$ over $\ZZ$ is called exceptional if the Gram matrix of the form $\chi$ with respect to this basis is upper unitriangular, i.e. $\chi(E_i,E_j) = 0$ for $i>j$ and $\chi(e_i, e_j) = 1$ for $i = j$.
\end{definition}

A standard example of an exceptional basis of $K_0(\PP_n)$ is $(\OO, \OO(1), \ldots, \OO(n))$. This statement follows from the Beilinson theorem, see \cite{B}; in fact, any $n+1$ consecutive invertible sheaves form an exceptional basis of $K_0(\PP_n)$.

    The braid group $B_{n+1}$ on $n+1$ strands acts on the set of all exceptional bases of the lattice $K_0(\PP_n)$. Namely, if $(E_0, E_1, \ldots, E_n$) is an exceptional basis of $K_0(\PP_n)$, then the inverse generators $g_i$ and $g_i^{-1}$ of $B_{n+1}$ replace the pair $E_i, E_{i+1}$ with $E_{i+1}-\chi(E_i, E_{i+1})E_i, E_i$ and $E_{i+1}, E_i - \chi(E_i, E_{i+1})E_{i+1}$ respectively and preserve all other vectors of the basis. A straightforward computation shows that these formulas agree with the generating relations $g_ig_{i+1}g_i = g_{i+1}g_ig_{i+1}$ for all $i$ and $g_ig_j=g_jg_i$ for $|i-j|>1$.\\

One of the central conjectures related to this action is the following:

\begin{conjecture}
    The group spanned by mutations of exceptional bases and the isometries of $K_0(\PP_n)$ acts transitively on the set of exceptional bases of $K_0(\PP_n)$.
\end{conjecture}

This conjecture is proved only for $n=2$ (see \cite{GR}) and $n=3$ (see \cite{N}). The arguments presented in \cite{GR, N} are unlikely to be generalized to the case of general $n$. Moreover, the proof for $n=3$ given in \cite{N} seems a little artificial, so it would also be interesting to give another proof of Conjecture 1 even in this particular case. For this purpose, it can be useful to study the group of isometries of $K_0(\PP_n)$ with respect to the Euler form. We denote this group by $\Isom(K_0(\PP_n))$.

In this paper, we prove several results about this group that help to understand its structure more clearly:
\begin{itemize}
    \item in Section 3.1 we prove that any element $\varphi\in \Isom_e(K_0(\PP_n)\otimes \RR)$ can be written as $E\mapsto E\otimes F_\varphi$ for some $F_\varphi \in K_0(\PP_n)\otimes \RR$, where $\Isom_e(K_0(\PP_n)\otimes \RR)$ is the identity component of the group $\Isom(K_0(\PP_n)\otimes \RR)$;
    \item in Section 3.2 we show that $\Isom(K_0(\PP_n))$ contains a subgroup of index 2 (denoted by $\Isom_e(K_0(\PP_n))$) isomorphic to $\ZZ^{[\frac{n+1}{2}]}$;
    \item in Section 3.3 we obtain a full description of $Isom(K_0(\PP_n))$ for $n\leqslant 6$, i.e. we find generators of this group explicitly.
\end{itemize}

\section*{Acknowledgements}\

The author is grateful to his academic supervisor Alexey Gorodentsev for posing the problem and permanent support.

\section{Preliminaries}\

Recall that $(\OO, \OO(1), \ldots, \OO(n))$ is an exceptional basis of $K_0(\PP_n)$. The Gram matrix of the form $\chi$ with respect to this basis is the following:
\[\begin{pmatrix}
    1 & n + 1 & \binom{n+2}{2} & \dots  & \binom{2n-1}{n-1} & \binom{2n}{n}\\
    0 & 1 & n + 1 & \dots  & \binom{2n-2}{n-2} & \binom{2n-1}{n-1} \\
    \vdots & \vdots & \vdots & \ddots & \vdots & \vdots \\
    0 & 0 & 0 & \dots & 1 & n + 1\\
    0 & 0 & 0 & \dots & 0 & 1
\end{pmatrix},\]
since $\chi(\OO(k),\OO(m)) = \binom{n+m-k}{n}$ for all $m,k\in \ZZ$.

Another basis of $K_0(\PP_n)$ which sometimes is more convenient for us to work with is $(\OO_{\PP_n}, \OO_{\PP_{n-1}}, \ldots \OO_{\PP_1}, \OO_{\PP_0})$, where $\OO_Y$ is the restriction of the structure sheaf $\OO$ to a subvariety $Y\subset \PP_n$. The basis $(\OO, \OO(1), \ldots, \OO(n))$ is expressed through $(\OO_{\PP_n}, \OO_{\PP_{n-1}}, \ldots \OO_{\PP_1}, \OO_{\PP_0})$ as follows:

$$\OO = \OO_{\PP_n},$$
$$\OO(1) = \OO_{\PP_n} + \OO_{\PP_{n-1}} + \ldots + \OO_{\PP_0};$$
$$\OO(2) = \OO_{\PP_n} + 2\OO_{\PP_{n-1}} + \ldots + (n+1)\OO_{\PP_0};$$
$$\ldots$$
$$\OO(n) = \OO_{\PP_n} + \binom{n+1}{n}\OO_{\PP_{n-1}} + \ldots + \binom{2n}{n}\OO_{\PP_0}.$$

The lattice $K_0(\PP_n)$ can be canonically identified with the module $$M_n = \{f\in\QQ[t]: f(\ZZ)\subset \ZZ\}\subset\QQ[t]$$ of integer valued polynomials of degree not greater than $n$ with rational coefficients. The identification is done via the isomorphism sending $[E]\in K_0(\PP_n)$ to its Hilbert polynomial~$\chi(E(t))$. Under this map, the basis $(\OO, \OO(1), \ldots, \OO(n))$ is identified with 
$$(\gamma_n(t), \gamma_n(t+1), \ldots, \gamma_n(t+n)),$$
where $\gamma_k(t) = \binom{t+k}{k} = \frac{1}{k!}(t+1)(t+2)\ldots (t+k)$.

We write $D$ for the differential operator $\frac{d}{dt}$ acting on $\QQ[t]$. Since the identification $M_n\cong K_0(\PP_n)$ can be extended to the identification $\QQ[t]_{\leqslant n}\cong K_0(\PP_n)\otimes_{\ZZ}\QQ$, we obtain an operator on $K_0(\PP_n)\otimes_{\ZZ}\QQ$ induced by $D$. This induced operator is also denoted by $D$.

Under the identification $M_n\cong K_0(\PP_n)$, the twisting $T\colon E\mapsto E(m) = E\otimes \OO(m)$ goes to the shift $e^{mD}\colon h_E(t)\mapsto h_E(t+m)$.

It follows from the exact sequence (here $s$ is a linear form vanishing along the hyperplane $\PP_{n-1}$)

$$0\to \OO(-1) \xrightarrow[]{\text{s}} \OO \to \OO\restr{\PP_{n-1}}\to 0$$
that for a locally free sheaf $E$ we have the following exact sequence:

$$0\to E(-1) \xrightarrow[]{\text{s}} E \to E\restr{\PP_{n-1}}\to 0.$$

We conclude that $h_{E\restr{\PP_{n-1}}} = h_E(t) - h_E(t-1) = \nabla h_E(t)$ for all $E\in K_0(\PP_n)$, where $$\nabla = 1 - e^{-D}.$$

In particular, the basis $(\OO_{\PP_n}, \OO_{\PP_{n-1}}, \ldots \OO_{\PP_1}, \OO_{\PP_0})$ goes to the basis 
$$(\gamma_n(t), \gamma_{n-1}(t),\ldots \gamma_0(t))$$
under the canonical identification of $K_0(\PP_n)$ with $M_n$.

Finally, it is worth saying a few words about the multiplicative structure on $M_n$ provided by tensor multiplication on $K_0(\PP_n)$. There is an isomorphism of $\QQ$-algebras $\QQ[[\nabla]]/\nabla^{n+1} = \QQ[[D]]/D^{n+1} \cong M_n\otimes_{\ZZ}\QQ$ given by $\psi(D) \mapsto \psi(D)\gamma_n$. The preimage of $M_n$ under this isomorphism is the lattice $\ZZ[[\nabla]]/\nabla^{n+1}$. If we identify this lattice with $K_0(\PP_n)$, the tensor product of classes of sheaves in $K_0(\PP_n)$ turns to the standard multiplication in $\ZZ[[\nabla]]/\nabla^{n+1}$. However, in terms of $M_n$ itself this multiplicative structure is much less clear.

\section{Main results}

\subsection{Isometries of $K_0(\PP_n)\otimes \RR$}\

We begin with a proposition describing the group of isometries of $K_0(\PP_n)\otimes \RR$.

\begin{proposition}
    The group $\Isom(K_0(\PP_n)\otimes \RR)$ as a subgroup of $GL(n+1,\RR)$ is abelian and has two connected components. The component of the identity $\Isom_e(K_0(\PP_n)\otimes \RR)$ is isomorphic to $\RR^{[\frac{n+1}{2}]}$. If we identify $K_0(\PP_n)\otimes_{\ZZ}\RR$ with $\RR[t]_{\leqslant n}$, then $$(\exp(D), \exp(D^3), \ldots, \exp(D^{2[\frac{n+1}{2}]-1}))$$ is a basis of $Isom_e(K_0(\PP_n)\otimes \RR)$.
\end{proposition}

A proof of Proposition 1 can be found in \cite{G}. In order to make the paper self-contained, we sketch this proof below.

Let $V$ be a finite dimensional vector space equipped with a nondegenerate bilinear form $\beta$. Then there exists a unique linear operator $\varkappa$ such that $\beta(u,v) = \beta(v, \varkappa u)$ for any $u,v\in V$. If $G$ is the Gram matrix of the form $\beta$ in some basis of the vector space $V$, then the matrix of $\varkappa$ is given by $G^{-1}G^T$. The operator $\varkappa$ is called the canonical operator of the form $\beta$.

Recall that given a linear operator $\varphi\colon V\to V$, its left (respectively, right) adjoint is the linear operator on $V$ denoted by $^{\vee}\varphi$ (respectively, $\varphi^{\vee}$) and uniquely determined by the formula $\beta(^{\vee}\varphi u, v) = \beta(u, \varphi v)$ (respectively, $\beta(u, \varphi^{\vee}v) = \beta(\varphi u, v)$) for any $u,v\in V$. In general $^{\vee}\varphi \ne \varphi^{\vee}$. However, a direct calculation shows that the following proposition holds.

\begin{proposition}
    In the previous notation, the following conditions are equivalent:\\
    $$(1)\quad ^{\vee}\varphi = \varphi^{\vee},\quad (2)\quad ^{\vee \vee}\varphi = \varphi, \quad (3)\quad \varphi^{\vee\vee} = \varphi, \quad (4)\quad \varkappa\varphi = \varphi \varkappa.$$
\end{proposition}
    

An operator $\varphi$ satisfying any of the equivalent conditions from Proposition 2 is called reflexive. We denote by $\mathcal A$ the algebra of reflexive operators.

\begin{definition}
    A linear operator $\varphi$ is called antiselfadjoint if $^{\vee}\varphi = \varphi^{\vee} = -\varphi$.
\end{definition}

Clearly, all antiselfadjoint operators are reflexive. The subspace of $\mathcal A$ consisting of antiselfadjoint operators will be denoted by $\mathcal A^-$.

The next proposition describes the Lie algebra of the algebraic group of isometries of~$(V,\beta)$.

\begin{proposition}
    The Lie algebra $\textnormal{Lie}(\Isom)$ of the group of isometries of $(V,\beta)$ is naturally isomorphic to $\mathcal A^-$. 
\end{proposition}

Proposition 3 can be proved via the arguments given in \cite{S} (Chapter 1, Theorem 5). Note that this proposition generalizes the well-known result saying that the Lie algebra of the orthogonal group $O(n)$ is the space of skew-symmetric matrices of size $n$.

Now, let us pass to the situation when $V = K_0(\PP_n)\otimes \RR$ and $\beta$ is the Euler form $\chi$. We want to show that $D$ is an antiselfadjoint operator. For this purpose, we need the following lemma.

\begin{lemma}
    The canonical operator of $\chi$ is $\varkappa = (-1)^{n}e^{-(n+1)D}\colon f(t) \mapsto (-1)^{n}f(t-n-1).$
\end{lemma}

\begin{proof}
    It follows from the Serre duality:
    $$\Ext^q(E,F) \cong \Ext^{n-q}(F, E(-n-1))^*$$
    that $\chi(E,F) = (-1)^n\chi(F, E\otimes \omega)$. This means that the canonical operator $\varkappa$ is $$(-1)^ne^{-(n+1)D}\colon f(t) \mapsto (-1)^nf(t-n-1).$$
    
    In other way this follows from the straightforward computation for the standard basis $\OO, \OO(1), \ldots \OO(n)$ of $K_0(\PP_n)$:
    
    $$\chi(\OO(m), \OO(k-n-1)) = \binom{n + k - n -1 - m}{n}= $$ $$ = \binom{k-m-1}{n}=\frac{(k-m-1)(k-m-2)\ldots(k-m-n)}{n!} = $$
    $$= (-1)^{n}\frac{(m-k+1)(m-k+2)\ldots(m-k+n)}{n!} = (-1)^{n}\binom{n+m-k}{n} =  $$ $$= (-1)^n\chi(\OO(k), \OO(m)),$$
    so we are done. 
\end{proof}

So, $\varkappa = (-1)^n Id + \eta$, where $\eta = (-1)^n(e^{-(n+1)D} - 1)$. It is easy to see that $\eta^{n+1} = 0$ but $\eta^n \ne 0$. So, the Jordan normal form of $\varkappa$ consists of only one block with eigenvalue~$(-1)^n$. Therefore, any operator commuting with $\varkappa$ is a polynomial in $\varkappa$, hence it is a polynomial in $D$.

Now, $D$ commutes with $\varkappa$, so $^{\vee}D = D^{\vee} = F(D)$, where $F\in \RR[x]$ is some polynomial. Since $D$ is nilpotent, $F(D)$ must also be nilpotent, i.e. the constant term of $F$ is $0$. It follows that $\varkappa^\vee = e^{-(n+1)F(D)}$. At the same time, $\varkappa\in \Isom(K_0(\PP_n))$, because $$\chi(E,F) = \chi(F, E\otimes \omega) = \chi(E\otimes \omega, F\otimes \omega).$$
Hence $^\vee\varkappa=\varkappa^\vee = \varkappa^{-1}$. This means that $e^{(n+1)F(D)} = e^{-(n+1)D}$, which can be rewritten as~$e^{-(n+1)(D+F(D))} = Id$. But $ D+ F(D)$ is nilpotent, so we must have $D+F(D)=0$ and $D^\vee = -D$.

Any antiselfadjoint operator commutes with $\varkappa$, so, again, it must be a polynomial in $\varkappa$. From this and from the fact that $D^\vee = -D$ we conclude that any antiselfadjoint operator is a linear combination of odd powers of the operator $D$. By exponentiating we conclude that the group $\Isom_e(K_0(\PP_n))$ is isomorphic to $\RR^{[\frac{n+1}{2}]}$ and $\exp(D), \exp(D^3), \ldots, \exp(D^{2[\frac{n+1}{2}]-1})$ is a basis of $\Isom_e(K_0(\PP_n\otimes \RR))$.

To finish the proof of Proposition 1, it remains to show that the group $\Isom(K_0(\PP_n))$ has two connected components. In order to do this, notice that $\Isom(K_0(\PP_n\otimes \RR))$ can be viewed as the subspace of the canonical algebra $\mathcal A$ given by $f(-D)f(D) = 1$, where $f(D) = a_0 + a_1D + a_2D^2 + \ldots + a_nD^n$, $a_i\in\RR$. The condition $f(-D)f(D) = 1$ can be rewritten explicitly as the following system:

\begin{equation*}
    \begin{cases}
      a_0^2 = 1,\\
      2a_0a_2 = a_1^2,\\
      2a_0a_4 = -2a_1a_3 - a_2^2,\\
      \ldots\\
    \end{cases}\,
\end{equation*}

We see that if we choose $a_0 = \pm 1$ and arbitrary values for odd coefficients $a_i$, then the remaining even coefficients are determined uniquely. So, the group $\Isom(K_0(\PP_n))$ has two connected components corresponding to the two possible choices of $a_0$. This ends the proof of Proposition 1.\\

So, any element of $\Isom_e(K_0(\PP_n)\otimes \RR)$ can be written as $\exp(a_1D + a_3D^3 + \ldots)$. This observation points out a way to find the elements of $\Isom_e(K_0(\PP_n))$, namely we need to choose $a_1, a_3, \ldots$ in such a way that $\exp(a_1D + a_3D^3 + \ldots)$ maps the lattice $K_0(\PP_n)$ bijectively to itself. We will work with this in the next chapter. Now, we state another proposition which allows us to express isometries of $K_0(\PP_n)$ and $K_0(\PP_n)\otimes \RR$ as tensor multiplication by some classes from $K_0(\PP_n)$.

\begin{proposition}
    Any element $\varphi \in \Isom_e(K_0(\PP_n)\otimes \RR)$ can be written as $E\mapsto E\otimes F_\varphi$ for some $F_\varphi \in K_0(\PP_n)\otimes \RR$.
\end{proposition}

Our proof of Proposition 4 is based on two lemmas.

\begin{lemma}
    $D(\OO_{\PP_k}) = \OO_{\PP_{k-1}} + \frac{1}{2}\OO_{\PP_{k-2}} + \frac{1}{3}\OO_{\PP_{k-3}} + \ldots + \frac{1}{k}\OO_{\PP_0}$ for any $k=0,1,
    \ldots, n$.
\end{lemma}

\begin{proof}
    Recall that the operator $\nabla = 1 - e^{-D}$ acts in the following way: $\nabla(\gamma_m) = \gamma_{m-1}$ for any $n\geqslant m\geqslant 1$ and $\nabla(\gamma_0) = 0$. So, $D = - \ln(1-\nabla) = \nabla + \frac{\nabla^2}{2} + \frac{\nabla^3}{3} + \ldots$ viewed as a formal power series. Since $\gamma_m$ is identified with $\OO_{\PP_m}$, this immediately implies our lemma.
\end{proof}

\begin{lemma}
    $\OO(k)\otimes \OO_{\PP_m} = \OO_{\PP_m} + k\OO_{\PP_{m-1}} + \binom{k+1}{2}\OO_{\PP_{m-2}} + \ldots + \binom{k+m-1}{m}\OO_{\PP_0}$ for any $k\in \ZZ$.
\end{lemma}

\begin{proof}
    Recall that in terms of polynomials multiplication by $\OO(k)$ is given by $$f(t)\mapsto f(t+k).$$ So, our formula can be rewritten as $$\binom{t+m+k}{m} = \binom{t+m}{m} + \binom{k}{1}\binom{t+m-1}{m-1} + \binom{k+1}{2}\binom{t+m-2}{m-2} + \ldots + \binom{k+m-1}{m}.$$ But it is a well-known identity from elementary combinatorics.
\end{proof}

\begin{proof}[Proof of Proposition 4]
    We are going to prove that this property is satisfied if we replace $\varphi$ with $D$.
    
    Let us show that $D(E) = E\otimes (\OO_{\PP_{n-1}} + \frac{1}{2}\OO_{\PP_{n-2}} + \frac{1}{3}\OO_{\PP_{n-3}} + \ldots + \frac{1}{n}\OO_{\PP_0})$ for any class~$E\in K_0(\PP_n\otimes \RR).$ Since the right hand side is linear with respect to $E$, it suffices to consider the case $E=\OO(k)$. The case $k=0$ was treated in Lemma 2. Now, let us prove it for arbitrary $k$.
    
    We have
    $$\OO(k) = \OO_{\PP_n} + k\OO_{\PP_{n-1}} + \binom{k+1}{2}\OO_{\PP_{n-2}} + \ldots + \binom{k+n-1}{n}\OO_{\PP_0},$$
    so
    
\begin{multline*}
    D(\OO(k)) = (\OO_{\PP_{n-1}} + \frac{1}{2}\OO_{\PP_{n-2}} + \ldots + \frac{1}{n}\OO_{\PP_0}) + k(\OO_{\PP_{n-2}} + \frac{1}{2}\OO_{\PP_{n-3}} + \ldots + \frac{1}{n-1}\OO_{\PP_0})+\ldots\\
    \ldots + \binom{k+n-2}{n-1}(\OO_{\PP_1}+\frac{1}{2}\OO_{\PP_0}) + \binom{k+n-1}{n}\OO_{\PP_0} =\\
    = \OO_{\PP_{n-1}} + (k + \frac{1}{2})\OO_{\PP_{n-2}} + \ldots + (\frac{1}{r} + \frac{k}{r-1} + \frac{1}{r-2}\binom{k+1}{2} + \ldots + \binom{k+r-2}{r-1})\OO_{\PP_{n-r}} + \ldots
\end{multline*}

On the other hand,
\begin{multline*}\OO(k)\otimes (\OO_{\PP_{n-1}} + \frac{1}{2}\OO_{\PP_{n-2}} + \frac{1}{3}\OO_{\PP_{n-3}} + \ldots + \frac{1}{n}\OO_{\PP_0}) = \\
= (\OO_{\PP_{n-1}} + k\OO_{\PP_{n-2}} + \binom{k+1}{2}\OO_{\PP_{n-3}} + \ldots + \binom{k+n-2}{n-1}\OO_{\PP_0}) +\\
+\frac{1}{2}(\OO_{\PP_{n-2}} + k\OO_{\PP_{n-3}} + \binom{k+1}{2}\OO_{\PP_{n-4}} + \ldots\\
\ldots + \binom{k+n-3}{n-2}\OO_{\PP_0}) + \ldots + \frac{1}{n-1}(k\OO_{\PP_1} + \OO_{\PP_0}) + \frac{1}{n}\OO_{\PP_0}=\\
=\OO_{\PP_{n-1}} + (k + \frac{1}{2})\OO_{\PP_{n-2}} + \ldots + (\frac{1}{r} + \frac{k}{r-1} + \frac{1}{r-2}\binom{k+1}{2} + \ldots + \binom{k+r-2}{r-1})\OO_{\PP_{n-r}} + \ldots
\end{multline*}

    We see that for any $i$ the coefficients of $\OO_{\PP_i}$ in the two expressions coincide. Thus, $D(\OO(k)) = \OO(k)\otimes (\OO_{\PP_{n-1}} + \frac{1}{2}\OO_{\PP_{n-2}} + \frac{1}{3}\OO_{\PP_{n-3}} + \ldots + \frac{1}{n}\OO_{\PP_0})$. This is exactly what we wanted to prove.

    Now, for any positive integer $m$ the following formula holds: $$D^m(E) = E\otimes (\OO_{\PP_{n-1}} + \frac{1}{2}\OO_{\PP_{n-2}} + \frac{1}{3}\OO_{\PP_{n-3}} + \ldots + \frac{1}{n}\OO_{\PP_0})^{\otimes m}.$$ From this we conclude that if $f$ is a polynomial, then $f(D)(E) = E\otimes F_{f(D)}$ for some $F_{f(D)} \in K_0(\PP_n)\otimes \RR$ and for any $E\in K_0(\PP_n)\otimes \RR$.
\end{proof}

\subsection{Isometries of $K_0(\PP_n)$: general results}\

Now, we are going to study the group $\Isom(K_0(\PP_n))$. It is clearly a subgroup of the \\ group $\Isom(K_0(\PP_n)\otimes \RR)$. Let us denote by $\Isom_e(K_0(\PP_n))$ the intersection $$\Isom(K_0(\PP_n))\cap \Isom_e(K_0(\PP_n)\otimes \RR).$$

\begin{lemma}
    $\Isom_e(K_0(\PP_n))$ is a subgroup of index $2$ of the group $\Isom(K_0(\PP_n))$.
\end{lemma}

\begin{proof}
    We only need to show that the set $\Isom(K_0(\PP_n))\setminus \Isom_e(K_0(\PP_n))$ is nonempty. This follows from the fact that the map $E \mapsto -E$ is an isometry of the lattice $K_0(\PP_n)$ that cannot be written in the form $e^{F(D)}$, where $F$ is a polynomial, because $e^{F(D)}$ has only one eigenvalue equal to $1$ of multiplicity $n+1$.
\end{proof}

From this moment on, we will only work with the subgroup $$\Isom_e(K_0(\PP_n))\subset \Isom(K_0(\PP_n)).$$

For this group, the following result similar to Proposition 1 holds.

\begin{proposition}
    The group $\Isom_e(K_0(\mathbb P_n))$ is isomorphic to $\mathbb Z^{[\frac{n+1}{2}]}$.
\end{proposition}

\begin{proof} $\Isom_e(K_0(\mathbb P_n))$ is clearly a discrete subgroup of $\Isom_e(K_0(\PP_n)\otimes \RR)$. Thus, it is a lattice, i.e. it is isomorphic to $\ZZ^k$ for some integer $k$. We need to show that~$k = [\frac{n+1}{2}]$. So, it suffices to prove that one can choose $[\frac{n+1}{2}]$ linearly independent elements in $\Isom_e(K_0(\mathbb P_n))$.

We claim that such elements can be chosen in the form $$\exp(b_1D),\quad \exp(b_2D^3), \quad \ldots, \quad \exp({b_{[\frac{n+1}{2}]}D^{2[\frac{n+1}{2}]-1}})$$ for some integers $b_i$. Indeed, any power of $D$ is a nilpotent operator, so the sum $$\exp({b_iD^{2i-1}}) = I + b_iD^{2i-1} + \frac{b_i^2}{2}D^{4i - 2} + \ldots$$ is actually finite, so we can take any integer $b_i$ such that all coefficients in this sum become integer after multiplication by $b_i$ (for example, we can choose $b_i = n!$).

If we choose $b_i$ this way, the operators $\exp({b_iD^{2i-1}})$ map the lattice $K_0(\PP_n)$ into itself. Surjectivity of these maps follows from the fact that the determinant of $\exp({b_iD^{2i-1}})$ equals~$1$, so the restriction of this operator to our lattice is invertible. This completes the proof of the proposition.

\end{proof}

\begin{remark}
    We would like to stress that the operators $e^{n!D^{2i-1}}$ do not necessarily form a basis of $\Isom(K_0(\PP_n))$. Moreover, we will see later that for some $n$ it is not even possible to choose $b_i$ in such a way that $\exp({b_iD^{2i-1}})$ generate the lattice $\Isom(K_0(\PP_n))$. Vice versa, there are elements of the form $\exp({a_1D + a_3D^3 + \ldots})$ lying in $\Isom(K_0(\PP_n))$ such that not all coefficients $a_i$ are integer.
\end{remark}

As we have already said, in order to compute the elements of $\Isom(K_0(\PP_n))$, it is sufficient to choose $a_1, a_2, \ldots\in \RR$ such that the operator $e^{a_1D + a_2D^3 + \ldots}$ maps the lattice $K_0(\PP_n))$ bijectively to itself. As we showed in the proof of Proposition 5, bijectivity is automatic once the lattice is mapped to itself. Nevertheless, this does not immediately allow us to compute all elements of $\Isom(K_0(\PP_n))$ explicitly for general $n$. However, we can prove the following assertion which is useful for computation of the isometry group~$\Isom_e(K_0(\PP_n))$.

\begin{proposition}
    If $e^{a_1D + a_2D^3 + a_3D^5 + \ldots}\in \Isom_e(K_0(\PP_n))$ for some $a_1, a_2, \ldots \in \RR$, then~$a_1\in~\ZZ$. Also, in this case $e^{b_1D + a_2D^3 + a_3D^5 + \ldots}\in \Isom_e(K_0(\PP_n))$ for any $b_1\in \ZZ$.
\end{proposition}

\begin{proof}
    $$e^{a_1D + a_2D^3 + a_3D^5 + \ldots} = e^{a_1D}e^{a_2D^3}e^{a_3D^5}\ldots = $$
    $$= (I+a_1D +\frac{a_1^2}{2}D^2 +\ldots)(I+a_2D^3 + \frac{a_2^2}{2}D^6 + \ldots)(I + a_3D^5 + \frac{a_3^2}{2}D^{10} + \ldots)\ldots =$$ $$ = I + a_1D + \frac{a_1^2}{2}D^2 + \ldots,$$ i.e. the coefficient of $D$ in the expression $D$ $e^{a_1D + a_2D^3 + a_3D^5 + \ldots}$ is $a_1$. This means that the polynomial $t+1$ (which corresponds to the class $\OO_{\PP_1}\in K_0(\PP_n)$) is sent by this operator to $1 + (1+a_1)t$. So, $a_1\in \ZZ$.
    
    At the same time, if $c\in \ZZ$, then $e^{cD}$ is an isometry of the lattice $K_0(\PP_n)$, because $e^{cD}$ maps the polynomial $f(t)$ to $f(t+c)$. Hence, for any $b_1\in \ZZ$ we have $e^{b_1D + a_2D^3 + a_3D^5 + \ldots} = e^{a_1D + a_2D^3 + a_3D^5 + \ldots}\cdot e^{(b_1-a_1)D}$, so the second part of the proposition is also proved. 
\end{proof}

\begin{remark}
In practice, this proposition means that when computing the isometries of $K_0(\PP_n)$ we can without loss of generality assume that $a_1 = 0$.
\end{remark}

\begin{proposition}
    If the operator $e^{a_1D + a_2D^3 + a_3D^5 + \ldots}$ maps $\OO$ to some element of the lattice $K_0(\PP_n)$, then $e^{a_1D + a_2D^3 + a_3D^5 + \ldots}\in \Isom_e(K_0(\PP_n))$.
\end{proposition}

\begin{proof}
    By Proposition 5 $\varphi = e^{a_1D + a_2D^3 + a_3D^5 + \ldots}$ maps any $E\in K_0(\PP_n)$ to $E\otimes F_\varphi$ for some element $F_\varphi \in K_0(\PP_n)\otimes \RR$. In particular, $\OO$ is mapped to $\OO\otimes F_\varphi = F_\varphi$. So, if $\OO$ is mapped to an element of $K_0(\PP_n)$, then $F_\varphi\in K_0(\PP_n)$, hence $E\otimes F_\varphi\in K_0(\PP_n)$ for any~$E\in K_0(\PP_n)$.
\end{proof}

Proposition 7 allows us to simplify computation of coefficients $a_1, a_2, \ldots$ satisfying the condition $e^{a_1D + a_2D^3 + a_3D^5 + \ldots}\in \Isom_e(K_0(\PP_n))$. Without this result, it would be necessary to check that each of the elements $\OO_{\PP_0}, \OO_{\PP_1}, \ldots \OO_{\PP_n}$ or some other basis of $K_0(\PP_n)$ is mapped to an element of the lattice $K_0(\PP_n)$ by the operator $e^{a_1D + a_2D^3 + a_3D^5 + \ldots}$. Proposition~7 shows that it is sufficient to check this only for $\OO_{\PP_n} = \OO$.

At the end of this section we discuss some special elements of $\Isom(K_0(\PP_n))$. One of such isometries is the operator $e^D$. In terms of integer valued polynomials, it acts as $f(t)\mapsto f(t+1)$. In terms of $K_0(\PP_n)$, it is the twisting that maps $E$ to $E(1) = E\otimes \OO(1)$.

Now, consider the case when all coefficients of $a_1D + a_2D^3 + a_3D^5 + \ldots$ are zero except for the last one, i.e. $a_{[\frac{n+1}{2}]}$. It looks differently for even and odd $n$.

Suppose that $n = 2k + 1$ is odd. Then $[\frac{n+1}{2}] = k$, $\exp({a_kD^{2k+1}}) = I + a_kD^{2k+1}$. The polynomial $\gamma_n = \binom{t+n}{n} = \frac{1}{n!}(t+1)(t+2)\ldots (t+n)$ is mapped to $\gamma_n(t) + a_k$ by the operator~$I + a_kD^{2k+1}$. So, $a_k\in\ZZ$. Conversely, if  $a_k\in\ZZ$, then we obtain an isometry of $K_0(\PP_n)$ by Proposition 7.

So, with respect to the basis $(\OO_{\PP_0}, \OO_{\PP_1}, \ldots \OO_{\PP_{n-1}}, \OO_{\PP_n})$, the isometry $\exp({D^{2k+1}})$ has the following matrix:

\[\begin{pmatrix}
    1 & 0 & \dots & 1\\
    0 & 1 & \dots  & 0 \\
    \vdots & \vdots & \ddots & \vdots \\
    0 & 0 & \dots & 1
\end{pmatrix}.\]

So, this isometry maps $\OO(m)$ to $\OO(m) + \OO_{\PP_0}$ for any $m$. So, it acts in the following way: $$E\mapsto E\otimes (\OO + \OO_{\PP_0}).$$

Next, suppose that $n = 2k$ is even. In this case, $[\frac{n+1}{2}] = k$, $\exp({a_kD^{2k-1}}) = I + a_kD^{2k-1}$. The polynomial $\gamma_{2k} = \binom{t+2k}{2k} = \frac{1}{(2k)!}(t+1)(t+2)\ldots (t+2k)$ is mapped to $\gamma_{2k}(t) + t + \frac{2k+1}{2}a_k$ by the operator $I + a_kD^{2k}$, so $a_k$ must be an even integer. Again, this condition is sufficient by Proposition 7.

With respect to the basis $(\OO_{\PP_0}, \OO_{\PP_1}, \ldots \OO_{\PP_{n-1}}, \OO_{\PP_n})$, the isometry $\exp({2D^{2k-1}})$ has the following matrix:

\[\begin{pmatrix}
    1 & 0 & \dots & 0 & 2 & 2k-1\\
    0 & 1 & \dots  & 0 & 0 & 2 \\
    \vdots & \vdots & \ddots & \vdots &\vdots & \vdots\\
    0 & 0 & \dots & 1 & 0 & 0\\
    0 & 0 & \dots & 0 & 1 & 0\\
    0 & 0 & \dots & 0 & 0 & 1
\end{pmatrix}.\]

So, this isometry maps $\OO(m)$ to $\OO(m) + 2\OO_{\PP_1} + (2k + 2m - 1)\OO_{\PP_0}$. Similarly to the case of odd $n$, we conclude that this isometry acts in the following way: $$E\mapsto E\otimes(\OO + 2\OO_{\PP_1} + (2k - 1)\OO_{\PP_0}).$$

The previous discussion can be summarized in the following proposition.

\begin{proposition}
\begin{enumerate}
    \item If $n\geqslant 4$ is even, then $\exp({a_{[\frac{n+1}{2}]}D^{2[\frac{n+1}{2}] - 1}})\in \Isom_e(K_0(\PP_n))$ if and only if $a_{[\frac{n+1}{2}]}$ is an even integer. The isometry $\exp({2D^{2[\frac{n+1}{2}] - 1}})$ acts as $$E\mapsto E\otimes(\OO + 2\OO_{\PP_1} + (n - 1)\OO_{\PP_0}).$$
    
    \item If $n\geqslant 3$ is odd, then the operator $\exp({a_{[\frac{n+1}{2}]}D^{2[\frac{n+1}{2}] - 1}})\in \Isom_e(K_0(\PP_n))$ if and only if $a_{[\frac{n+1}{2}]}$ is an integer. The isometry $\exp({D^{2[\frac{n+1}{2}] - 1}})$ acts as $$E\mapsto E\otimes(\OO + \OO_{\PP_0}).$$

\end{enumerate}
\end{proposition}

The isometries from Proposition 8 admit the following alternative decription. For odd $n = 2k + 1$ the operator $\exp({D^{2k+1}})$ can be written as
$$E\mapsto E + rk(E)\OO_{\PP_0}$$

and for even $n = 2k$ the operator $\exp({2D^{2k-1}})$ can be written as
$$E \mapsto E + 2rk(E)\OO_{\PP_1} + (2c_1(E) + (2k-1)rk(E))\OO_{\PP_0},$$
where $rk(E)$ and $c_1(E)$ are the rank and the first Chern class of $E\in K_0(\PP_n)$ respectively. These formulas follow from the fact that they hold for $E = \OO(m)$ and from linearity of the rank and the first Chern class.

Note that this alternative description is possible in this particular case because the coefficients of $\OO_{\PP_{m-1}}$ and $\OO_{\PP_m}$ in the expression for $\OO(k)\otimes \OO_{\PP_m}$ from Lemma 3 are linear polynomials in $k$. So, the whole expression is linear in $k$, and we are done by linearity of the rank and the first Chern class. However, in general case the coefficients in the expression $\OO_{\PP_m} + k\OO_{\PP_{m-1}} + \binom{k+1}{2}\OO_{\PP_{m-2}} + \ldots + \binom{k+m-1}{m}\OO_{\PP_0}$ are polynomials in $k$ of degree greater than $1$. This is why such an alternative way of writing an isometry is not possible generally.

\subsection{Isometries of $K_0(\PP_n)$: explicit formulas for small $n$}\

Here we formulate several propositions that describe the group $K_0(\PP_n)$ for $n\leqslant 6$ in explicit terms. The simplest cases are $n=1$ and $n=2$.

\begin{proposition}
    For $n=1,2$ the group $\Isom_e(K_0(\PP_n))$ is isomorphic to $\ZZ$ and is generated by the twisting $E\mapsto E(1) = E\otimes \OO(1)$. 
\end{proposition}

\begin{proof}
    For $n=1,2$ we have $[\frac{n+1}{2}] = 1$, so $\Isom_e(K_0(\PP_n))$ is indeed isomorphic to $\ZZ$ and generated by $e^{aD}$ for some $a\in \RR$. From Proposition 6 it follows that $a = 1$. As we have already said, $e^D$ is the twisting $E\mapsto E\otimes \OO(1)$. This completes the proof.
\end{proof}

The situation becomes more interesting when $n=3$. We have the following result.

\begin{proposition}
    The group $\Isom_e(K_0(\PP_3))$ is isomorphic to $\ZZ^2$. It is generated by the operators $E \mapsto E\otimes \OO(1)$ and $E\mapsto E\otimes (\OO + \OO_{\PP_0})$.
\end{proposition}

\begin{proof}
    Any element of the group $\Isom_e(K_0(\PP_3))$ can be written as $e^{aD + bD^3}$, где $a,b\in\RR$. We know that $e^D$ is an isometry acting as $E \mapsto E\otimes \OO(1)$. To find other isometries, we can assume that $a = 0$ (here we again use Proposition 6). But we have already shown that for $n=2k-1$  we have $e^{bD^{2k-1}}\in \Isom_e(K_0(\PP_{2k-1}))$ only for $b\in \ZZ$. This gives us a generator of the group $\Isom_e(K_0(\PP_3))$ equal to $e^{D^3}$ and given by the formula $E\mapsto E\otimes(\OO + \OO_{\PP_0})$.
\end{proof}

For $n=4$, we have a similar proposition.

\begin{proposition}
    The group $\Isom_e(K_0(\PP_4))$ is isomorphic to $\ZZ^2$. It is generated by the operators $E \mapsto E\otimes \OO(1)$ and $E\mapsto E\otimes (\OO + 2\OO_{\PP_1} + 3\OO_{\PP_0})$.
\end{proposition}

\begin{proof}
    The proof is similar to the proof of the previous proposition. The only difference is that we need to apply the first part of Proposition~8 from the previous section.
\end{proof}

For $n\geqslant 5$, the situation becomes more difficult. As an example, we will describe $\Isom_e(K_0(\PP_n))$ explicitly for $n=5,6$.

\begin{proposition}
    The group $\Isom_e(K_0(\PP_5))$ is isomorphic to  $\ZZ^3$ and is generated by the operators
    \begin{align*}
    & E\mapsto E\otimes \OO(1),\\
    & E \mapsto E\otimes(\OO + 2\OO_{\PP_2} + 3\OO_{\PP_1} + 4\OO_{\PP_0}),\\
    & E\mapsto E\otimes (\OO + \OO_{\PP_0}).
    \end{align*}
\end{proposition}

\begin{proof}
    Again, the fact that $\Isom_e(K_0(\PP_5))\cong \ZZ^3$ follows from Proposition 5. Applying Proposition 6, we find the generator $E\mapsto E\otimes \OO(1)$, while others are of the form $e^{bD^3 + cD^5}$, where $b,c\in \RR$. Since $D^6 = 0$, we obtain the equality $$e^{bD^3 + cD^5} = (I + bD^3)(I+cD^5) = I + bD^3 + cD^5.$$
    
    The polynomial $\gamma_5(t) = \frac{(t+1)(t+2)(t+3)(t+4)(t+5)}{120}$ is mapped to $$\gamma_5(t) + b(\frac{t^2}{2} + 3t + \frac{17}{4}) + c = \gamma_5(t) + \frac{b}{2}t^2 + 3bt + (\frac{17b}{4} + c) = \gamma_5(t) + b\gamma_2(t) + \frac{3b}{2}\gamma_1(t) + (\frac{7b}{4} + c)$$ by the operator $I + bD^3 + cD^5$. So, $e^{bD^3 + cD^5}$ is an isometry if and only if the numbers $b$, $\frac{3b}{2}$, $\frac{7b}{4} + c$ are integers. From the fact that $\frac{3b}{2}\in \ZZ$ it follows that $b$ is even. Also, $\frac{7b}{4} + c\in \ZZ$, so either $b\in 2\ZZ\setminus 4\ZZ$ and $c\in \frac{1}{2}\ZZ\setminus \ZZ$ or $b\in 4\ZZ$ and $c\in \ZZ$. This gives us two generators of the group $\Isom_e(K_0(\PP_5))$, namely $I + 2D^3 + \frac{1}{2}D^5$ и $I + D^5$.
    
    The isometry $I + 2D^3 + \frac{1}{2}D^5$ maps $\gamma_5(t)$ to $\gamma_5(t) + 2\gamma_2(t) + 3\gamma_4(t) + 4\gamma_0(t)$. In terms of classes in $K_0(\PP_n)$, it maps $\OO$ to $\OO + 2\OO_{\PP_2} + 3\OO_{\PP_1} + 4\OO_{\PP_0}$. By Proposition 4, this isometry acts as $$E \mapsto E\otimes(\OO + 2\OO_{\PP_2} + 3\OO_{\PP_1} + 4\OO_{\PP_0}).$$
    
    As we already know, the isometry $I + D^5$ acts in the following way: $E\mapsto E\otimes (\OO + \OO_{\PP_0})$. This completes the proof.
\end{proof}

\begin{remark}
    Note that we cannot choose three generators of $\Isom_e(K_0(\PP_5))$ in the form  $e^{aD}$, $e^{bD^3}$, $e^{cD^5}$ for some $a,b,c\in \RR$.
\end{remark}

\begin{proposition}
    The group $\Isom_e(K_0(\PP_6))$ is isomorphic to $\ZZ^3$ and is generated by the operators
\begin{align*} 
& E\mapsto E\otimes \OO(1),\\
& E\mapsto E\otimes (\OO + 2\OO_{\PP_3} + 3\OO_{\PP_2} + 4\OO_{\PP_1} + 7\OO_{\PP_0}),\\
& E\mapsto E\otimes (\OO + 2\OO_{\PP_1} +  5\OO_{\PP_0}).
\end{align*}
\end{proposition}

\begin{proof}
    As before, we have the isometry $E\mapsto E\otimes \OO(1)$ while others can be written in the form $e^{bD^3 + cD^5}$, где $b,c\in \RR$. In our case $D^7 = 0$, so $$e^{bD^3 + cD^5} = (I + bD^3 + \frac{b^2}{2}D^6)(I+cD^5) = I + bD^3 + cD^5 + \frac{b^2}{2}D^6.$$ The operator $I + bD^3 + cD^5 + \frac{b^2}{2}D^6$ maps the polynomial $\gamma_6(t) = \frac{(t+1)(t+2)\ldots (t+6)}{720}$ to $$\gamma_6(t) + b(\frac{t^3}{6} + \frac{7t^2}{4} + \frac{35t}{6} + \frac{49}{8}) + c(t+\frac{7}{2}) + \frac{b^2}{2} = $$
    $$ = \gamma_6(t) + \frac{b}{6}t^3 + \frac{7b}{4}t^2 + (\frac{35b}{6} + c)t + (\frac{49b}{8} + \frac{7c}{2} + \frac{b^2}{2}) = $$
    $$ = \gamma_6(t) + b\gamma_3(t) + \frac{3b}{2}\gamma_2(t) + (\frac{7b}{4} + c)\gamma_1(t) + (\frac{15b}{8} + \frac{b^2}{2} + \frac{5c}{2}).$$ Therefore, the numbers $b$, $\frac{3b}{2}$, $\frac{7b}{4} + c$, $\frac{15b}{8} + \frac{b^2}{2} + \frac{5c}{2}$ must be integers. Since $\frac{3b}{2}$ is an integer, $b$ is even and hence $\frac{b^2}{2}$ is an integer. Thus, the numbers $\frac{7b}{4} + c$ and $\frac{15b}{8} + \frac{5c}{2}$ are also integers. For even $b$ there are two possibilities: either $b\in 2\ZZ\setminus 4\ZZ$ and $c\in \frac{1}{2} + 2\ZZ$, or $b\in 4\ZZ\setminus 8\ZZ$ and $c\in 1 + 2\ZZ$, or $b\in 8\ZZ$ and $c\in 2\ZZ$. The operators $$I + 2D^3 + \frac{1}{2}D^5 + 2D^6 (b=2, c=\frac{1}{2})$$ and $$I + 2D^5 (b=0, c = 2)$$ can be chosen as generators of the lattice.
    
    The operator $I + 2D^3 + \frac{1}{2}D^5 + 2D^6$ maps $\gamma_6(t)$ to $\gamma_6(t) + 2\gamma_3(t) + 3\gamma_4(t) + 4\gamma_1(t) + 7$. In other words, it maps $\OO$ to $\OO + 2\OO_{\PP_3} + 3\OO_{\PP_2} + 4\OO_{\PP_1} + 7\OO_{\PP_0}$. So, this isometry acts as $$E\mapsto E\otimes (\OO + 2\OO_{\PP_3} + 3\OO_{\PP_2} + 4\OO_{\PP_1} + 7\OO_{\PP_0}).$$
    
    The fact that the operator $I + 2D^5$ acts as $E\mapsto E\otimes (\OO + 5\OO_{\PP_0})$ was already proved in the previous section.
\end{proof}

\Addresses

\end{document}